\documentclass{mrlart7}     
\pdfoutput=1   
\usepackage{pb-diagram,enumerate,amsmath}     
\usepackage{graphicx}
\usepackage{color}

\theoremstyle{plain}      
 
\newtheorem{thm}{Theorem}[section]     
\newtheorem{theorem}[thm]{Theorem}     
     
\newtheorem{corollary}[thm]{Corollary}     
     
\newtheorem{lemma}[thm]{Lemma}     
\newtheorem{prop}[thm]{Proposition}     
\newtheorem{proposition}[thm]{Proposition}

\theoremstyle{remark}

\theoremstyle{definition}

\renewcommand{\epsilon}{\varepsilon}

\def\al{{\alpha}}

\def\om{{\omega}}

\def\si{{\sigma}}
\def\Si{{\Sigma}}

\def\ep{{\varepsilon}}

\def\th{{\vartheta}}

\def\phi{{\varphi}}

\let\pa\partial
\let\na\nabla

\let\theta\vartheta
\let\phi\varphi
     
\def\cE{{\mathcal E}}

\def\cT{\mathcal{T}}

\DeclareMathAlphabet{\doba}{U}{msb}{m}{n}         
\gdef\mC{\doba{C}}

\gdef\mR{\doba{R}}

\gdef\mZ{\doba{Z}}

\def\Cl{{\mathop{\rm Cl}}}     
\def\End{{\mathop{\rm End}}}     
\def\divv{{\mathop{\rm div}}}     
\def\tr{{\mathop{\rm tr}}}     
\def\Id{\operatorname{Id}}    
\def\SO{\operatorname{SO}}
\def\GL{\operatorname{GL}}
\def\Spin{\operatorname{Spin}}    
\def\diag{\operatorname{diag}}    

\def\so{{\mathop{\mathfrak{so}}}}

\let\ti\tilde   

\def\eref#1{{\rm (\ref{#1})}}   

\let\<\langle 
\let\>\rangle 
\let\witi\widetilde

\newcommand{\definedas}{\mathrel{\raise.095ex\hbox{\rm :}\mkern-5.2mu=}}

  \def\volno{18}
  \def\yrno{2011}
  \def\issueno{00}
  \def\lpageno{100NN} 
\begin{document}     


\title{Harmonic spinors and local deformations of the metric}

\author{Bernd Ammann} 
\address{Bernd Ammann, Fakult\"at f\"ur Mathematik \\ 
Universit\"at Regensburg \\
93040 Regensburg \\  
Germany}
\email{bernd.ammann@mathematik.uni-regensburg.de}
\urladdr{http://www.berndammann.de}

\author{Mattias Dahl} 
\address{Mattias Dahl, Institutionen f\"or Matematik \\
Kungliga Tekniska H\"ogskolan \\
100 44 Stockholm \\
Sweden}
\email{dahl@math.kth.se}

\author{Emmanuel Humbert} 
\address{Emmanuel Humbert, Institut \'Elie Cartan, BP 239 \\ 
Universit\'e de Nancy 1 \\
54506 Vandoeuvre-l\`es-Nancy Cedex \\ 
France}
\email{humbert@iecn.u-nancy.fr}


\date{June 7, 2011}


\begin{abstract}
Let $(M,g)$ be a compact Riemannian spin manifold. The Atiyah-Singer 
index theorem yields a lower bound for the dimension of the kernel 
of the Dirac operator. We prove that this bound can be attained by 
changing the Riemannian metric $g$ on an arbitrarily small open set.  
\end{abstract}

\maketitle

\section{Introduction and statement of results}

Let $M$ be a spin manifold, we assume that all spin manifolds come
equipped with a choice of orientation and spin structure. The Dirac
operator $D^g$ of $(M,g)$ is a first order differential operator
acting on sections of the spinor bundle associated to the spin
structure on $M$. This is an elliptic, formally self-adjoint
operator. If $M$ is compact, then the spectrum of $D^g$ is real,
discrete, and the eigenvalues tend to plus and minus infinity. In this
case the operator $D^g$ is invertible if and only if $0$ is not an
eigenvalue, which is the same as vanishing of the kernel. 

The Atiyah-Singer Index Theorem states that the index of the Dirac
operator is equal to a topological invariant of the manifold,
\begin{equation*}
\operatorname{ind}(D^g) = \alpha(M),
\end{equation*}
see for example \cite[Theorem 16.6, p. 276]{Lawson_Michelsohn_89}.
Depending on the dimension $n$ of $M$ this formula has slightly
different interpretations. To explain this interpretation, it is
important to remark that we will always consider the spinor bundle as
a complex vector bundle, similar results with different dimensions
would also hold for the real spinor bundle or the 
${\rm C\ell}_n$-linear spinor bundle. If $n$ is even there is a
$\pm$-grading of the spinor bundle and the Dirac operator~$D^g$ has a
part $(D^g)^+$ which maps from positive to negative spinors. If 
$n \equiv 0, 4 \mod 8$ the index is integer valued and computed as the  
dimension of the kernel minus the dimension of the cokernel of 
$(D^g)^+$. If $n \equiv 1, 2 \mod 8$ the index is $\mZ/2\mZ$-valued 
and given by the dimension modulo $2$ of the kernel of~$D^g$
(if $n \equiv 1 \mod 8$) resp. $(D^g)^+$ (if $n \equiv 2 \mod 8$). In
other dimensions the index is zero. In all dimensions $\alpha(M)$ is a
topological invariant depending only on the spin bordism class of $M$.
In particular, $\alpha(M)$ does not depend on the metric, but it
depends on the spin structure in dimension $n\equiv 1,2 \mod 8$.
For further details see \cite[Chapter II, \S 7]{Lawson_Michelsohn_89}.

The index theorem implies a lower bound on the dimension of the kernel
of~$D^g$ which we can write succinctly as
\begin{equation}\label{lower.bound}
\dim \ker D^g \geq a(M),
\end{equation}
where
\begin{equation*}
a(M) \definedas
\begin{cases}
|\widehat{A}(M)|, &\text{if $n \equiv 0 \mod 4$;} \\
1, &\text{if $n \equiv 1 \mod 8$ and $\alpha(M)\neq 0$;} \\
2, &\text{if $n \equiv 2 \mod 8$ and $\alpha(M)\neq 0$;} \\
0, &\text{otherwise.} \\ 
\end{cases}.
\end{equation*}
If $M$ is not connected, then this lower bound can be improved by
studying each connected component of $M$. For this reason we restrict
to connected manifolds from now on.
 
Metrics $g$ for which equality holds in \eref{lower.bound} are called
$D$-minimal, see \cite[Section~3]{Baer_Dahl_02}. 
The existence of $D$-minimal metrics on all connected compact spin 
manifolds was established in \cite{Ammann_Dahl_Humbert_09} 
following previous work 
in \cite{Maier_97} and \cite{Baer_Dahl_02}. In this note we will 
strengthen this existence result by showing that one can find a 
$D$-minimal metric coinciding with a given metric outside a small 
open set. For a Riemannian manifold $(M,g)$ we denote by $U_p(r)$ the
set of points for which the distance to the point $p$ is strictly less
than $r$. We will prove the following theorem.

\begin{theorem} \label{main_thm}
Let $(M,g)$ be a compact connected Riemannian spin manifold of
dimension $n \geq 2$. Let $p \in M$ and $r > 0$. Then there is a
$D$-minimal metric $\widetilde{g}$ on~$M$ with $\widetilde{g} = g$ 
on~$M \setminus U_{p}(r)$.
\end{theorem}

The new ingredient in the proof of this theorem is the use of the 
``invertible double'' construction which gives a $D$-minimal metric 
on any spin manifold of the type $(-M) \# M$ where $\#$ denotes 
connected sum and where $-M$ denotes $M$ equipped with the 
opposite orientation. 
For dimension $n \geq 5$ we can then use the surgery
method from \cite{Baer_Dahl_02} with surgeries of codimension 
$\geq 3$. For $n = 3, 4$ we need the stronger surgery result of 
\cite{Ammann_Dahl_Humbert_09} preserving $D$-minimality under
surgeries of codimension $\geq 2$. The case $n=2$ follows from
\cite{Ammann_Dahl_Humbert_09} and classical facts about Riemann
surfaces.

If a manifold has one $D$-minimal metric then generic metrics are
$D$-minimal, to formulate this precisely we introduce some notation. 
We denote by $\mathcal{R}(M,U_{p}(r),g)$ the set of all 
smooth Riemannian metrics on $M$ which coincide with the metric 
$g$ outside $U_{p}(r)$ and by $\mathcal{R}_{\rm min}(M, U_{p}(r), g)$ 
the subset of $D$-minimal metrics. From Theorem \ref{main_thm} it
follows that a generic metric from $\mathcal{R}(M , U_{p}(r) , g)$ is
actually an element of $\mathcal{R}_{\rm min}(M , U_{p}(r) , g)$, as
made precise in the following corollary.

\begin{corollary} \label{main_cor}
Let $(M,g)$ be a compact connected Riemannian spin manifold of 
dimension $\geq 3$. Let $p \in M$ and $r > 0$. Then
$\mathcal{R}_{\rm min}(M , U_{p}(r) , g)$ is open in the 
$C^1$-topology on $\mathcal{R}(M , U_{p}(r) , g)$ and it is 
dense in all $C^k$-topologies, $k \geq 1$. 
\end{corollary}
The proof follows \cite[Theorem 1.2]{Anghel_96} or 
\cite[Proposition 3.1]{Maier_97}. The first observation of the
argument is that the eigenvalues of $D^g$ are continuous functions of
$g$ in the $C^1$-topology, from which the property of being open
follows. The second observation is that spectral data of $D^{g_t}$
for a linear family of metrics $g_t = (1-t) g_0 + t g_1$ depends real
analytically on the parameter $t$. If 
$g_0 \in \mathcal{R}_{\rm min}(M, U_{p}(r), g)$ it follows that
metrics arbitrarily close to $g_1$ are also in this set, from which 
we conclude the property of being dense.

\section{Preliminaries}

\subsection{Spin manifolds and spin structure preserving maps}

An orientation on an $n$-dimensional manifold $M$ can be viewed as a
refinement of the frame bundle $\GL(M)$ for the tangent bundle $TM$ to
a sub-bundle $\GL_+(M)$ with structure group $\GL_+(n,\mR)$. Such a
refinement exists if and only if the first Stiefel-Whitney class
$w_1(TM)$ vanishes. Here the group $\GL_+(n,\mR)$ consists of all
invertible $n \times n$-matrices with positive determinant and has
fundamental group $\mZ$ if $n=2$ and $\mZ/2\mZ$ if $n \geq 3$. Let 
$\widetilde{\GL_+}(n,\mR)$ be the unique connected double cover of
$\GL_+(n,\mR)$.

A (topological) spin structure on an oriented manifold $M$ is a 
($\widetilde{\GL_+}(n,\mR) \to \GL(n,\mR)$)-equivariant lift of
$\GL_+(M)$ to a bundle with structure group
$\widetilde{\GL_+}(n,\mR)$. Such a lift exists if and only if the
second Stiefel-Whitney class $W_2(TM)$) vanishes.

If these structures exist they are in general not unique, the
orientation can be chosen independently on each connected
component of $M$, or equivalently the space of orientations on $M$ is
an affine space for the $\mZ/2\mZ$-vector space $H^0(M,\mZ/2\mZ)$.
Similarly, the space of spin structures is an affine space 
for the $\mZ/2\mZ$-vector space $H^1(M,\mZ/2\mZ)$.

As already mentioned we use the term ``spin manifold'' for a manifold 
together with the choice of an orientation and a spin structure. 

If $f:M_1 \to M_2$ is a diffeomorphism between two manifolds, any
orientation and spin structure on $M_2$ pulls back to an orientation
and spin structure on $M_1$. A~diffeomorphism $f$ between two spin
manifolds $M_1$ and $M_2$ is called a spin structure preserving
diffeomorphism if the orientation and spin structure on $M_1$ coincide
with the pullbacks from $M_2$.

If the manifold $M$ is further equipped with a Riemannian metric the
above topological spin structure reduces to a geometrical
spin structure which is a ($\Spin(n) \to \SO(n)$)-equivariant lift
$\Spin(M)$ of the bundle $\SO(M)$ of oriented orthonormal frames of
the tangent bundle. The spinor bundle $\Si M$ on $M$ is a vector
bundle associated to $\Spin(M)$, it has a natural first order elliptic
operator $D:\Gamma(\Sigma M) \to \Gamma(\Sigma M)$, see for example 
\cite{friedrich_00} for details. Any spin structure preserving
diffeomorphism $f:M_1 \to M_2$ which is also an isometry induces an
isomorphism between the spinor bundles $f_*:\Si M_1 \to \Si M_2$ which
is compatible with the Dirac operators in the sense that all sections
$\phi$ of $\Si M_1$ satisfy $D^{M_2}(f_* \circ \phi \circ f^{-1}) 
= f_* \circ (D^{M_1}\phi)\circ f^{-1}$.

If $W$ is a manifold with boundary $\pa W = M$, then an orientation
and spin structure on $W$ induce an orientation and a spin structure
on $M$. Conversely, if an orientation and a spin structure on $M$ are
given, then there is a unique orientation and spin structure on 
$W = M \times [0,1]$ such that the restricted structures on 
$M \cong M \times \{1\}$ coincide with the given ones. The boundary
component $M\times \{0\}$ is obviously diffeomorphic to $M$ as well,
but the restriction of the orientation of $M \times [0,1]$ is the
opposite of the orientation of $M$. We write 
$-M \definedas M \times \{0\}$ for the spin manifold with this
opposite orientation and the spin structure obtained from 
$M \times [0,1]$.

\subsection{The invertible double}

Let $N$ be a compact connected spin manifold with boundary. The double
of $N$ is formed by gluing $N$ and $-N$ along the common boundary 
$\partial N$ and is denoted by $(-N) \cup_{\partial N} N$. If $N$ is
equipped with a Riemannian metric which has product structure near the
boundary, then this metric naturally gives a metric on 
$(-N) \cup_{\partial N} N$. The spin structures can be glued together
to obtain a spin structure on $(-N) \cup_{\partial N} N$. The spinor
bundle  of $(-N) \cup_{\partial N} N$ is obtained by gluing the spinor
bundle of $N$ with the spinor bundle of $-N$ along their common
boundary $\partial N$. It is straightforward
to check that the appropriate gluing map is the map used in 
\cite[Chapter 9]{Booss_Wojciechowski_93}.

The Dirac operator on $(-N) \cup_{\partial N} N$ is invertible 
due to the following argument. Assume that a spinor field $\phi$ is in
the kernel of the Dirac operator on $(-N) \cup_{\partial N} N$. The
restriction $\phi|_{-N}$ can be ``reflected along $\pa N$'' to a
spinor field $\tilde \phi$ on $N$ as indicated in the appendix. On the
boundary $\pa N$ one has $\ti\phi|_N=\nu\cdot \phi|_N$ and thus 
$\nu\cdot \ti\phi|_N=-\phi|_N$ for the exterior unit normal
field~$\nu$ on $\pa N$, see Lemma~\ref{lemma.a}. Green's formula for
the Dirac operator yields 
\begin{equation*}
0 = 
\int_N \<D\ti \phi, \phi\> - \int_N \<\ti\phi,D\phi\>
= \int_{\pa N} \<\nu\cdot \ti\phi,\phi\>
= -\|\phi|_{\pa N}\|_{L^2(\pa N)}^2.
\end{equation*}
Thus $\phi|_{\pa N}= 0$, and by the weak unique continuation property
of the Dirac operator it follows that $\phi = 0$. For more details on
this argument see \cite[Chapter 9]{Booss_Wojciechowski_93} and 
\cite[Proposition 1.4]{Booss_Lesch_08}.
In  the appendix we also show that the doubling construction of
\cite[Chapter 9]{Booss_Wojciechowski_93} coincides with the spinor
bundle and Dirac operator on the doubled manifold.


\begin{prop} \label{prop_double}
Let $(M,g)$ be a compact connected Riemannian spin manifold. 
Let $p \in M$ and $r > 0$. Let $(-M) \# M$ be the connected sum formed
at the points $p \in M$ and $p \in -M$. Then there is a metric on 
$(-M) \# M$ with invertible Dirac operator which coincides with $g$
outside $U_{p}(r)$
\end{prop}

This Proposition is proved by applying the double construction to the
manifold with boundary $N = M \setminus U_{p}(r/2)$, where
$N$ is equipped with a metric we get by deforming the metric $g$ on 
$U_{p}(r) \setminus U_{p}(r/2)$ to become product near the boundary.

Metrics with invertible Dirac operator are obviously $D$-minimal, so
the metric provided by Proposition \ref{prop_double} is $D$-minimal.

\section{Proof of Theorem \ref{main_thm}}

Let $M$ and $N$ be compact spin manifolds of dimension $n$. Recall
that a spin bordism from $M$ to $N$ is a manifold with boundary $W$ of
dimension $n+1$ together with a spin structure preserving
diffeomorphism from $N \amalg (-M)$ to the boundary of $W$. The
manifolds $M$ and $N$ are said to be spin bordant if such a bordism
exists. 

For the proof of Theorem \ref{main_thm} we have to distinguish
several cases.

\subsection{Proof of Theorem \ref{main_thm} in dimension  $n \geq 5$}

\begin{proof}

To prove the Gromov-Lawson conjecture, Stolz \cite{Stolz_92}
showed that any compact spin manifold with vanishing index is spin
bordant to a manifold of positive scalar curvature. Using this we see 
that $M$ is spin bordant to a manifold $N$ which has a $D$-minimal
metric $h$, where the manifold $N$ is not necessarily connected. For
details see \cite[Proposition 3.9]{Baer_Dahl_02}.

\begin{center}
\includegraphics[width=4cm]{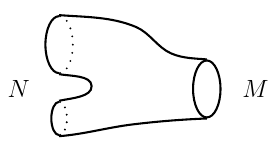}\kern1.6cm 
\end{center}

By removing an open ball from the interior of a spin bordism from $M$
to $N$ we get that $N \amalg (-M)$ is spin bordant to the sphere
$S^n$. 

\begin{center}
\includegraphics[width=6cm]{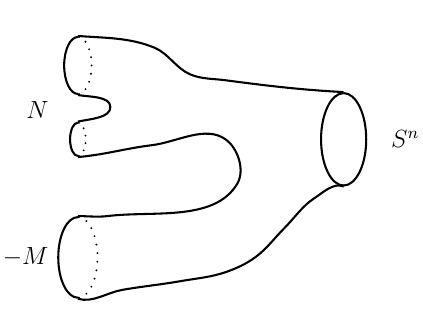} 
\end{center}

Since $S^n$ is simply connected and $n \geq 5$ it follows from
\cite[Proof of Theorem 4.4, page 300]{Lawson_Michelsohn_89} that 
$S^n$ can be obtained from $N \amalg (-M)$ by a sequence of surgeries 
of codimension at least $3$. By making $r$ smaller and possibly move
the surgery spheres slightly we may assume that no surgery hits 
$U_{p}(r) \subset M$. We obtain a sequence of manifolds 
$N_0, N_1, \dots, N_k$, where $N_0 = N \amalg (-M)$, $N_k = S^n$, 
and $N_{i+1}$ is obtained from $N_i$ by a surgery of codimension at 
least $3$.

\begin{center}
\includegraphics[width=6cm]{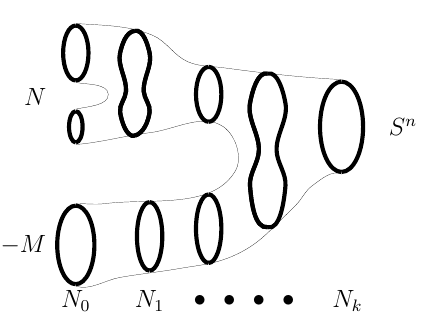} 
\end{center}

Since the surgeries do not hit $U_{p}(r) \subset M \subset N
\amalg (-M) = N_0$ we can consider $U_p(r)$ as a subset of every
$N_i$. We define the sequence of manifolds $N'_0, N'_1, \dots,
N'_k$ by forming the connected sum $N'_i = M \# N_i$ at the points
$p$. Then $N'_0 = N \amalg (-M) \# M$, $N'_k = S^n \# M = M$, and 
$N'_{i+1}$ is obtained from $N'_i$ by a surgery of codimension at 
least $3$ which does not hit $M \setminus U_{p}(r)$.

\begin{center}
\includegraphics[width=7cm]{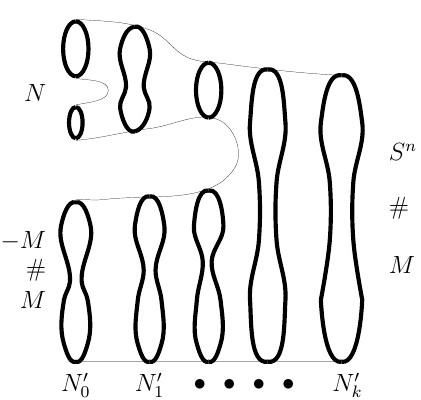} 
\end{center}

We now equip $N'_0$ with a Riemannian metric. On $N$ we choose a
$D$-minimal metric. The manifold $(-M) \# M$ has vanishing index, 
so a $D$-minimal metric is a metric with invertible Dirac operator. 
From Proposition \ref{prop_double} we know that there exists such a
metric on $(-M) \# M$ which coincides with $g$ outside
$U_{p}(r)$. Note that here we use the assumption that $M$ is
connected. Together we get a $D$-minimal metric $g'_0$ on $N'_0$.

From \cite[Proposition 3.6]{Baer_Dahl_02} we know that the property of
being $D$-minimal is preserved under surgery of codimension at least
$3$. We apply the surgery procedure to~$g'_0$ to produce a sequence of
$D$-minimal metrics $g'_i$ on $N'_i$. Since the surgery procedure of 
\cite[Theorem 1.2]{Baer_Dahl_02} does not affect the Riemannian
metrics outside arbitrarily small neighborhoods of the surgery spheres
we may assume that all $g'_i$ coincide with~$g$ on 
$M \setminus U_{p}(r)$. The Theorem is proved by choosing 
$\tilde{g} = g'_k$ on $N'_k = M$.
\end{proof}

\subsection{Proof of Theorem \ref{main_thm} in dimensions  $n=3$ 
and $n=4$} 

\begin{proof}
In these cases the argument
works almost the same, except that we can only conclude that $S^n$ is
obtained from $N \amalg (-M)$ by surgeries of codimension at least $2$,
see \cite[VII, Theorem 3]{Kirby_89} for $n=3$ and 
\cite[VIII, Proposition 3.1]{Kosinski_93} for $n=4$. To take care of 
surgeries of codimension $2$ we use 
\cite[Theorem 1.2]{Ammann_Dahl_Humbert_09}. Since this surgery
construction affects the Riemannian metric only in a small
neighborhood of the surgery sphere we can finish the proof as
described in the case $n\geq 5$.
\end{proof}
  
Alternatively, it is straight-forward to adapt the perturbation proof
by Maier~\cite{Maier_97} to prove Theorem \ref{main_thm} in dimensions
$3$ and $4$.

\subsection{Proof of Theorem \ref{main_thm} in dimension  $n=2$}

\begin{proof}
The argument in the case $n=2$ is different. Assume that a metric $g$
on a compact surface with chosen spin structure is given. 
In~\cite[Theorem 1.1]{Ammann_Dahl_Humbert_09} it is shown that for any
$\ep>0$ there is a $D$-minimal metric $\hat g$ with  
$\|g - \hat g\|_{C^1}<\ep$. Using the following Lemma \ref{dim2.lem},
we see that for $\ep>0$ sufficiently small, there is a spin structure 
preserving diffeomorphism $\psi:M\to M$ 
such that $\ti g \definedas \psi^*\hat g$
is conformal to $g$ on $M\setminus U_p(r)$. As the dimension of the
kernel of the Dirac operator is preserved under spin structure
preserving conformal diffeomorphisms, $\ti g$ is $D$-minimal as well.
\end{proof}

\begin{lemma}\label{dim2.lem}
Let $M$ be a compact surface with a Riemannian metric $g$ and a spin 
structure. Then for any $r>0$ there is an $\ep>0$ with the following 
property: For any $\hat g$ with $\|g - \hat g\|_{C^1}<\ep$ there is a
spin structure preserving diffeomorphism $\psi:M\to M$ such that 
$\ti g \definedas \psi^*\hat g$ is conformal to $g$ on
$M\setminus U_p(r)$.
\end{lemma}

To prove the lemma one has to show that a certain differential is
surjective. This proof can be carried out in different mathematical
languages. One alternative is to use Teichm\"uller theory formulated
in terms of quadratic differentials, we will use a presentation in
terms of Riemannian metrics following \cite{Tromba_92}.

\begin{proof}[Sketch of Proof of Lemma \ref{dim2.lem}]
If $g_1$ and $g_2$ are metrics on $M$, then we say that $g_1$ is 
Teichm\"uller equivalent to $g_2$ if there is a diffeomorphism 
$\psi: M \to M$ such that~$\psi$ is
homotopic to the identity and $\psi^* g_2$ is conformal to $g_1$. 
This is an equivalence relation on the set of metrics on $M$, and the
equivalence class of $g_1$ is denoted by $\Phi(g_1)$.  Let $\cT$ be
the set of equivalence classes, this is the Teichm\"uller space which
has a natural structure of a smooth finite-dimensional
manifold. 
Note that any diffeomorphism $\psi: M\to M$ homotopic to the identity is also 
isotopic to the identity, 
i.e.\ the homotopy can be chosen as a path in the diffeomorphism
group, see e.g. \cite{earle.eells:67}. As along this path, 
the spin structure is preserved, $\psi$ perserves the spin structure.

Showing the lemma is thus equivalent to showing that 
$\Phi(\mathcal{R}(M ,U_{p}(r),g))$ is a neighborhood of $\Phi(g)$ in
$\cT$. 

Variations of metrics are given by symmetric $(2,0)$-tensors, that is
by sections of $S^2 T^*M$. The tangent space of $\cT$ can be
identified with the space of transverse (= divergence free) traceless 
sections,
\begin{equation*}
S^{TT}
\definedas
\{ h \in \Gamma(S^2 T^*M) \mid \divv^g h=0, \tr^g h=0 \},
\end{equation*}
see for example \cite[Lemma 4.57]{Besse_87} and \cite{Tromba_92}.

The two-dimensional manifold $M$ has a complex structure which is
denoted by~$J$. The map $H: T^*M \to S^2 T^*M$ defined by 
$H(\al) \definedas \al \otimes \al - \al\circ J \otimes \al\circ J$ 
is quadratic, it is $2$-to-$1$ outside the zero section, and its image
are the trace free symmetric tensors. Furthermore 
$H(\al \circ J) = -H(\al)$. Hence by polarization we obtain an
isomorphism of real vector bundles from $T^*M \otimes_\mC T^*M$ to the
trace free part of $S^2 T^*M$. Here the complex tensor product is used
when $T^* M$ is considered as a complex line bundle using $J$. A trace
free section of $S^2 T^*M$ is divergence free if and only if the
corresponding section $T^*M \otimes_\mC T^*M$ is holomorphic, see  
\cite[pages 45-46]{Tromba_92}. We get that $S^{TT}$ is
finite-dimensional, and it follows that $\cT$ is finite dimensional. 

In order to show that $\Phi(\mathcal{R}(M ,U_{p}(r),g))$ is a
neighborhood of $\Phi(g)$ in $\cT$ we show that the differential
$d\Phi: T\mathcal{R}(M ,U_{p}(r),g) \to T\cT$ is surjective at $g$.
Using the above identification $T\cT = S^{TT}$, $d\Phi$ is just
orthogonal projection from $\Gamma(S^2 T^*M)$ to $S^{TT}$.

Assume that $h_0 \in S^{TT}$ is orthogonal to 
$d\Phi(T\mathcal{R}(M ,U_{p}(r),g))$. Then $h_0$ is $L^2$-orthogonal
to $T\mathcal{R}(M ,U_{p}(r),g)$.
As $T\mathcal{R}(M ,U_{p}(r),g))$ consists of all sections of 
$S^2 T^*M$ with support in $U_{p}(r)$ we conclude that $h_0$ vanishes
on  
$U_{p}(r)$. Since $h_0$ can be identified with a holomorphic section 
of $T^*M \otimes_\mC T^*M$ we see that $h_0$ vanishes everywhere on
$M$. The surjectivity of $d\Phi$ and the lemma follow.
\end{proof}


\appendix

\section{Notes about reflections at hypersurfaces and the doubling
  construction} 

Let $M$ be a connected Riemannian spin manifold, with a reflection 
$\phi$ at a hyperplane $N$. That is $\phi$ is an isometry with 
fixed point set $N$, orientation reversing, and
$N$ separates $M$ into two components. Let $-M$ be the manifold $M$ 
with the opposite orientation, i.e. $\phi:M\to -M$ is orientation
preserving. It is also required that $\phi$ preserves the spin structure. 
The reflection $\phi$ lifts to the frame bundle by mapping the frame 
$\cE=(e_1,\ldots,e_n)$ to 
$\phi_*\cE\definedas(-d\phi(e_1),d\phi(e_2),\ldots,d\phi(e_n))$, 
so $\phi_*:\SO(M)\to \SO(M)$. 
This map $\phi_*$ is not $\SO(n)$ equivariant, but if we define
$J=\diag(-1,1,1,1,...1)$, then
\begin{equation*}
\phi_*(\cE A)
=
\phi_*(\cE)JAJ.
\end{equation*}
If $\cE$ is a frame over $N$ whose first vector is normal to $N$, 
then $\phi_*(\cE)=\cE$.

The above mentioned compatibility with the spin structure is the fact
that the pullback of the double covering $\th:\Spin(M)\to \SO(M)$ via 
$\phi_*$ is again the covering $\Spin(M)\to \SO(M)$. In other words, a
lift $\witi\phi_*:\Spin(M)\to \Spin(M)$ can be chosen
such that $\th\circ \witi\phi_*=\phi_*\circ \th$.
This implies that $(\witi\phi_*)^2=\pm \Id$.
Choose $\witi\cE\in \Spin(M)$ over $N$, such that the first vector of 
$\th(\ti\cE)$ is normal to $N$. Then $\witi\phi_*(\witi\cE)=\pm
\witi\cE$, thus $(\witi\phi_*)^2(\witi\cE) = \witi\cE$. It follows that 
$(\witi\phi_*)^2=\Id$. 

The conjugation with $J$ is an automorphism of $\SO(n)$ and lifts to  
$\Spin(n) \subset \Cl_n$, as a conjugation with $E_1 \definedas
(1,0...,0)$ in the Clifford algebra sense. We therefore have
\begin{equation*}
\witi\phi_*(\witi\cE B)
=  
\witi\phi_*(\witi\cE)(-E_1BE_1).
\end{equation*}

Let $\si:\Cl_n\to \End(\Si_n)$ be an irreducible representation 
of the Clifford algebra.
$\Si M\definedas\Spin(M)\times_\si \Si_n$.

\begin{lemma}[Lift to the spinor bundle]\label{lemma.a}
The map 
\begin{equation*}
\Spin(M) \times \Si_n \ni 
(\witi\cE,\rho) \mapsto (\witi\phi_*\witi\cE,\si(E_1)\rho)
\in \Spin(M) \times \Si_n
\end{equation*} 
is compatible with the equivalence relation given by $\si$. Thus
it descends to a map 
\begin{equation*}
\phi_\#:\Si M= \Spin(M)\times_\si \Si_n
\to 
\Si M= \Spin(M)\times_\si \Si_n.
\end{equation*}
\end{lemma}

\begin{proof}
 $(\witi\cE B,\si^{-1}(B)\rho)$ is mapped to 
\begin{equation*}
(\witi\phi_*(\witi\cE B),\si(E_1)\si^{-1}(B)\rho)
=  
\witi\phi_*(\witi\cE)(-E_1BE_1), \si((-E_1BE_1)^{-1})\si(E_1)\rho).
\end{equation*} 
\end{proof}

Obviously $(\phi_\#)^2=-\Id$, and $\phi_\#:\Si_pM \to \Si_{\phi(p)}M$.
In even dimensions $\phi_\#$ maps positive spinors to negative ones and vice 
versa.

\begin{lemma}[On the fixed point set $N$]\label{lemma.a}
Assume that $\psi\in \Si M|_N$. Then $\phi_\#(\psi)
= \pm \nu\cdot \psi$ for a unit normal vector $\nu$ of $N$ in $M$. 
The sign depends on the choice of $\nu$ and the choice of the
lift~$\ti\phi_*$.
\end{lemma}

\begin{proof}
Choose $\witi\cE\in \Spin(M)$ over the base point of $\psi$,
such that $\nu$ is the first vector of $\th(\ti\cE)$. Determine $\rho\in\Si_n$
such that $(\witi\cE,\rho)$ represents $\psi$. Then $\phi_\#(\psi)$ is represented
by $(\pm \witi\cE,\nu\cdot \rho)$. 
\end{proof}

\begin{lemma}[Compatibility with the Clifford action]
\begin{equation*}
d\phi(X)\cdot \phi_\#(\psi)
=
-\phi_\#(X\cdot \psi)
\end{equation*}
for $X\in T_pM$, $\psi\in \Si_pM$. 
\end{lemma}

\begin{proof}
We view $T_pM$ as an associated bundle to $\Spin(M)$. Then 
$d\phi([\witi\cE,v])=[\witi\phi_*(\cE), Jv]$. 
Thus
\begin{equation*}
\begin{split}
  d\phi([\witi\cE,v])\cdot \phi_\#([\witi\cE,\rho])
   &= [\witi\phi_*(\witi\cE), \si(Jv)\si(E_1)\rho] \\
   &=[\witi\phi_*(\witi\cE), -\si(E_1)\si(v)\rho]\\
   &= -\phi_\#([\witi\cE,v]\cdot [\witi\cE,\rho]).
\end{split}
\end{equation*}
Here we used that $Jv = E_1 \cdot v \cdot E_1$ in $\Cl_n$.
\end{proof}

\begin{lemma}
Let $X\in T_pM$, $\psi\in \Gamma(\Si M)$. 
Then
\begin{equation*}
\na_{d\phi(X)}\phi_\#(\psi)=\phi_\#(\na_X\psi).
\end{equation*}
\end{lemma}

\begin{proof}
The differential of $\phi_*:\SO(M)\to \SO(M)$ maps $T\SO(M)$ to
$T\SO(M)$. The connection $1$-form $\om:\SO(M)\to \so(n)$ then pulls
back according to 
\begin{equation*}
\om((d(\phi_*))(Y)) = J \om(Y) J
\end{equation*}
for $Y\in T_\cE\SO(M)$, a lift of $X\in T_M$ under the projection
$\SO(M)\to M$. We lift this to a connection $1$-form $\ti
\om:\Spin(M)\to \Cl_n$ which thus transforms as
\begin{equation*}
\ti\om ((d(\ti\phi_*))(\ti Y))
= -E_1 \om(\ti Y) E_1
\end{equation*}
where $\ti Y\in T\Spin(M)$ is a lift of $Y$. And this induces the
relation 
\begin{equation*}
\na_{d\phi(X)}\phi_\#(\psi)
= \phi_\#(\na_X\psi).
\end{equation*}
\end{proof}

We obtain 
\begin{equation*}
\begin{split}
\phi_\#(D\psi) 
&= 
\sum_i  \phi_\#(e_i\cdot \na_{e_i}\psi)\\
&=
- \sum_i d\phi(e_i)\cdot \phi_\#(\na_{e_i}\psi)\\
&= 
- \sum_i d\phi(e_i)\cdot\na_{d\phi(e_i)} \phi_\#\psi\\
&=
- D \phi_\#\psi
\end{split}
\end{equation*}

This formula can also be read as
\begin{equation}\label{formula.Dphifis}
D\psi = \phi_{\#} D  \phi_{\#} \psi
\end{equation}
As a conclusion we obtain the following proposition.

\begin{proposition}
If one constructs the double for a manifold with the classical spinor 
bundle and Dirac operator as in
\cite[Theorem~9.3]{Booss_Wojciechowski_93}, then we obtain the
classical spinor bundle and the classical Dirac operator on the
double.
\end{proposition}

To prove the proposition one has to compare the definitions in 
\cite{Booss_Wojciechowski_93} with ours. 
The map $\phi_\#:\Si_p^+M\to \Si_{\phi(p)}^-M$ corresponds 
to the map $G$ in \cite{Booss_Wojciechowski_93}. It follows
that $G^{-1}$ corresponds to 
$-\phi_\#:\Si_p^-M\to \Si_{\phi(p)}^+M$. 
In \cite{Booss_Wojciechowski_93}, the map $G$ is used to identify
$\Si_p^+M$ with $\Si_{\phi(p)}^-M$. 
Pay attention that with respect to this identification, the map
$\phi_\#:\Si_p^+M\to \Si_{\phi(p)}^-M$ is the identity, whereas
$\phi_\#:\Si_p^-M\to \Si_{\phi(p)}^+M$ is $-\Id$.
Equation~\eref{formula.Dphifis} says that this identification is 
compatible with the Dirac operator, and corresponds to (9.10) 
in \cite{Booss_Wojciechowski_93}.

\section*{Acknowledgements} 

We thank Martin M\"oller, Frankfurt, for providing a proof of 
Lemma~\ref{dim2.lem} using Teichm\"uller theory. His proof was an
inspiration for the argument in the case $n=2$ presented above. We
also thank the referee for many helpful suggestions.


\providecommand{\bysame}{\leavevmode\hbox to3em{\hrulefill}\thinspace}
\providecommand{\MR}{\relax\ifhmode\unskip\space\fi MR }
\providecommand{\MRhref}[2]{%
  \href{http://www.ams.org/mathscinet-getitem?mr=#1}{#2}
}
\providecommand{\href}[2]{#2}


\end{document}